\documentclass[12pt]{amsart}
\usepackage{epsfig,color}

\makeatother

\theoremstyle{definition}

\def\fnum{equation} 
\newtheorem{Thm}[\fnum]{Theorem}
\newtheorem{Cor}[\fnum]{Corollary}

\newtheorem{Lem}[\fnum]{Lemma}

\numberwithin{equation}{section}

\newcommand{\Ric}{{\text{Ric}}}

\newcommand{\Hess}{{\text {Hess}}}

\def\RR{{\bold R}}

\newcommand{\dv}{{\text {div}}}
\newcommand{\e}{{\text {e}}}

\newcommand{\cL}{{\mathcal{L}}}

\newcommand{\eqr}[1]{(\ref{#1})}

\title{Eigenvalue lower bounds and splitting for modified Ricci flow}

\author{Tobias Holck Colding}%
\address{MIT, Dept. of Math.\\
77 Massachusetts Avenue, Cambridge, MA 02139-4307.}
\author{William P. Minicozzi II}%

\thanks{The  authors
were partially supported by NSF  DMS Grants   2104349  and 2005345.}

\email{colding@math.mit.edu and minicozz@math.mit.edu}

\begin{document}

\maketitle

{\centering\footnotesize Dedicated to our friend Joel Spruck.\par}

\begin{abstract}
We  prove sharp lower bounds for eigenvalues of the drift Laplacian for a modified Ricci flow.  
The modified Ricci flow is a system of coupled equations for a metric and weighted volume that plays an important role in Ricci flow.   
We will also show that there is a splitting theorem in the case of equality.
 \end{abstract}
 

\section{Introduction}

A  metric $g$ and function $f$ on a manifold $M$ induce a  weighted $L^2$ norm
$\| u \|_{L^2}^2 = \int u^2 \, \e^{-f}$, a corresponding weighted energy, and a natural elliptic operator 
\begin{align}
	\cL = \Delta - \nabla_{\nabla f}
\end{align}
 called the 
     drift Laplacian.     When $M$ is compact, or $g$ is complete with a log Sobolev inequality and $\int \e^{-f} \, dv_g < \infty$, then
      $\cL$ has 
      eigenvalues
\begin{align}
	0 = \lambda_0 (t) < \lambda_1 (t) \leq \dots  \to \infty  
\end{align}
that carry important geometric information. 
     
   The triple
        $(M,g,f)$ is a gradient shrinking  Ricci soliton (or shrinker) if $g$ and $f$ satisfy
\begin{align}
	\Hess_f + \Ric = \frac{1}{2} \, g \, , 
\end{align}
where $\Ric$ is the Ricci  curvature  for the   metric $g$, \cite{Ca, H}.   Ricci shrinkers appear as singularities of Ricci flow.
  For shrinkers, there is a sharp lower bound
 $\lambda_1 \geq \frac{1}{2}$, where  equality is achieved if and only if the shrinker splits off a line, \cite{CxZ}.   
 In \cite{CM2}, we showed  that if a  shrinker almost splits on one scale, then 
 it also almost splits on larger scales; this was called
  propagation of almost splitting.  See \cite{ChL} for an asymptotic splitting theorem for shrinkers.
  
  We are interested here in an analogous lower bound and splitting for flows.  The parabolic analog of the shrinker equation is 
  the modified (or rescaled) Ricci flow.  In this case, the   metric $g$ and function $f$   evolve by
\begin{align}	\label{e:gf1}
	g_t &= g - 2 \, \Hess_f - 2 \, \Ric \, , \\
	f_t &= \frac{n}{2} - S - \Delta \, f \, .  \label{e:gf2}
\end{align} 
These coupled equations are important in Ricci flow since they  describe, up to scaling and diffeomorphisms, flows arising from blowing up at a  singularity; see, e.g., \cite{MM, SW}.   
They also arise  as the negative gradient flow for Perelman's entropy functional, \cite{P,T}. 
    A static solution  of \eqr{e:gf1} is a   shrinker.
    
    Easy examples show that there is no uniform lower bound for $\lambda_1$ on a modified Ricci flow, even if the flow is ancient, but instead some additional condition is necessary.  The next   theorems give generalizations  for modified Ricci flows.

\begin{Thm}		\label{c:lambda1bd}
If  $M$ is compact, $g,f$  satisfy \eqr{e:gf1} and \eqr{e:gf2} for $t \in [t_0 , \infty)$,    and there is a sequence $t_j \to \infty$ so that $(M, g(t_j), f(t_j))$ converges to a shrinker, then $\lambda_1 (t)  > \frac{1}{2}$ for all $t$.
\end{Thm}

Furthermore, we get   sharp upper bounds for the evolution of   $\lambda_k (t)$ for all of the eigenvalues:

\begin{Thm}	\label{l:compk}
If $M$ is compact, $g,f$  satisfy \eqr{e:gf1} and \eqr{e:gf2}, and
  $k \geq 1$, then:
\begin{itemize}
	\item If $\lambda_k (t_0) < \frac{1}{2}$, then $\lambda_k (t) < \lambda_k (t_0)$
	for $t_0 < t$ and 
	\begin{align}
	 \lambda_k (t) \leq \frac{\lambda_k(t_0)}{2\,\lambda_k(t_0)\,(1-\e^{t-t_0})+\e^{t-t_0}} \,  .\label{e:holdsk}
\end{align}
	\item If $\lambda_k (t_0) = \frac{1}{2}$, then 
	$\lambda_k (t) \leq \frac{1}{2}$ for 
	all $t_0 \leq t$.
\item If $\frac{1}{2} < \lambda_k (t_0)$, then  \eqr{e:holdsk} applies for all $t<t_0+\log \,\frac{2\,\lambda_k(t_0)}{2\,\lambda_k(t_0)-1}$.
\end{itemize}
\end{Thm}

\vskip1mm
The assumption $t<t_0+\log \,\frac{2\,\lambda_k(t_0)}{2\,\lambda_k(t_0)-1}$ in the last case is equivalent to assuming that the denominator in \eqr{e:holdsk}
is positive.

\vskip1mm
The previous theorems generalize to a wide class of non-compact weighted spaces,  requiring only that all integrations by parts involved make sense and  integrals converge.  We will assume this in the following theorem.

\begin{Thm}	\label{P:highsplit}
Suppose that \eqr{e:gf1} and \eqr{e:gf2} hold for $t \in [t_0 , \infty)$.
If   $\lambda_k (t_0) = \frac{1}{2}$ for some $k \geq 1$ and $\lambda_1 (t_1) \geq \frac{1}{2}$, then $g(t)$ splits off an $\RR^k$ factor  for $t \in [t_0 , t_1]$.  
\end{Thm}

 \vskip1mm
 When a modified flow arises from blowing up at a singularity, then the possible long-time limits of the flow  describe the possible tangent flows at the singularity.   Thus, the uniqueness of blowups question for 
a Ricci flow is translated into the uniqueness of limits for the modified Ricci flow, cf. \cite{CM1} for mean curvature flow and \cite{Ba, CMZ1,CMZ2} for uniqueness of closed blowups for Ricci flows.

\section{Differential inequalities}

The eigenvalues $\lambda_k$ of $\cL$ vary continuously in $t$, but are not necessarily differentiable.  Differentiability can fail because of the multiplicity of the eigenvalues.  However, the variational characterization of the eigenvalues gives a natural upper bound for $\lambda_k(t)$ in the future.  This bound implies an upper bound for $\lambda_k'(t)$ in the sense of the limsup of forward difference quotients.  We collect some elementary consequences of such upper bounds.

We will say that a continuous function $h(t)$ satisfies $h'(t) \leq G(t)$ for a continuous function $G$ if it does so in the sense of the limsup of forward difference quotients:
For each $t$, we have
\begin{align}	\label{e:lsup}
	\limsup_{0 < \delta \to 0} \, \frac{h(t+\delta) - h(t)}{\delta} \leq G(t) \, .
\end{align}
There is a corresponding chain rule: If $h$ satisfies \eqr{e:lsup} on $I=(t_0 - \epsilon , t_0 + \epsilon)$ and $\Phi$ is a $C^1$ function on $h(I)$ with $\Phi' ( h(t_0))>0$, then
\begin{align}	\label{e:lsupA}
	\limsup_{0 < \delta \to 0} \, \frac{\Phi \circ h(t_0+\delta) - \Phi \circ h(t_0)}{\delta} \leq \Phi' (h(t_0)) \, G(t_0) \, .
\end{align}

\begin{Lem}	\label{l:FTCsup}
If $h$ is continuous and $h'(t) \leq c \in \RR$ in the sense of \eqr{e:lsup} for $t \in [a,b]$, then $h(t) \leq h(a) + c \, (t-a)$ for every $t \in [a,b]$.
\end{Lem}

\begin{proof}
Let $\epsilon > 0$ be arbitrary and define 
\begin{align}
	I_{\epsilon} = \{ t_0 \in [a,b] \, | \, h(t) \leq h(a) + (c+\epsilon) \, (t-a) {\text{ for all }} t \leq t_0 \} \, .
\end{align}
Since $h$ is continuous, $I_{\epsilon}$ is closed.  Note that $a \in I_{\epsilon}$.  To see that $I_{\epsilon} = [a,b]$, we will show that it is open.  Suppose therefore that $t_0 \in I_{\epsilon}$ for some $t_0 < b$.  In particular, 
\begin{align}
	h(t_0) \leq h(a) + (c+\epsilon) \, (t_0-a)
\end{align}
 and, by assumption, $h'(t_0) \leq c$ in the sense of \eqr{e:lsup}.
It follows that there exists $\delta_0 > 0$ so that
\begin{align}
	 \frac{h(t+\delta) - h(t)}{\delta} \leq (c+\epsilon) {\text{ for every }} \delta \in (0 , \delta_0] \, .
\end{align}
It follows that $t_0 + \delta \in I_{\epsilon}$ and, thus, $I_{\epsilon}$ is open and equal to all of $[a,b]$.  Since this is true for every $\epsilon > 0$, we conclude that $h(t) \leq h(a) + c \, (t-a)$ for every $t \in [a,b]$.
\end{proof}

\begin{Lem}	\label{l:ftc}
Suppose that $h(t) \geq 0$ satisfies $h' (t) \leq h(t) \, (h(t) -1)$ in the sense of \eqr{e:lsup}.
\begin{itemize}
\item If $h(t_0) < 1$, then $h(t) < h (t_0)$ for all $t > t_0$.
\item If $h(t_0) = 1$, then $h(t) \leq 1$ for all $t \geq t_0$.
\end{itemize}
\end{Lem}

\begin{proof}
The first claim follows since $h(t_0) < 1$ implies that $h'(t_0) < 0$ which gives that $h(t_0 + \delta) < h(t_0) < 1$ for all $\delta > 0$ sufficiently small.  

We will prove the second claim by contradiction.  Suppose therefore that there exists $t_2 > t_0$ with $h(t_2) > 1$.  Let $t_1$ be the maximal value of
$t \in [t_0, t_2]$ with $h(t) = 1$ (this exists since $h$ is continuous and $h(t_0) = 1$).  It follows that $t_1 < t_2$, $h(t_1)=1$ and $h\geq 1$ on $[t_1,t_2]$.
In particular, $h' \leq 0$ on $[t_1,t_2]$.  Therefore, 
Lemma \ref{l:FTCsup} gives that $h\leq h(t_1)=1$ on $[t_1,t_2]$, contradicting that $h(t_2) > 1$ and, thus, completing the argument.
\end{proof}

\begin{Lem}	\label{l:diffineq}
Suppose that $F' \leq   ( 2 \, F -1) \, F$ in the sense of \eqr{e:lsup} for $t\geq t_0$.  
\begin{enumerate}
\item If $F (t_0) < \frac{1}{2}$, then  $F (t_1) < F (t_0) < \frac{1}{2}$ for all $t_1 > t_0$ and
\begin{align}
 F (t_1) \leq \frac{F(t_0)}{2\,F(t_0)\,(1-\e^{t_1-t_0})+\e^{t_1-t_0}} \,  . \label{e:case1}
\end{align}
\item If 
$\frac{1}{2}<F (t)$ for $t \in [t_0 , t_1]$ and $t_1<t_0+\log \,\frac{2\,F(t_0)}{2\,F(t_0)-1}$, then  \eqr{e:case1} holds for $t_1$.
\end{enumerate}
\end{Lem}

\begin{proof}
Set $h=2\,F$ so that $h'(t) \leq h\,(h-1)$.
   It follows from the chain rule  that
\begin{enumerate}
\item[(A)]  If $F(t) < \frac{1}{2}$, then $ h(t)< 1$ and \eqr{e:lsupA} with $\Phi (s) = \log \frac{s}{1-s}$ gives that $\left(\log \frac{h}{1-h}\right)'\leq -1$.
\item[(B)]  If $\frac{1}{2}<F(t)$, then $1 < h(t)$ and  \eqr{e:lsupA} with $\Phi (s) = \log \frac{s-1}{s}$ gives that $\left(\log \frac{h-1}{h}\right)'\leq 1$
\end{enumerate}
By Lemma \ref{l:ftc}, if $h(\bar{t}) < 1$ for any $\bar{t}$, then $h (t) < h(\bar{t}) < 1$ for every $t > \bar{t}$.  In particular, if $h(t_0) < 1$, then (A) applies on the entire interval and Lemma \ref{l:FTCsup} gives that
\begin{align}
 \log \frac{h(t_1)}{1-h(t_1)} \leq -(t_1 - t_0) + \log \frac{h(t_0)}{1-h(t_0)} 
     \, .
\end{align}
 Exponentiating this and using that $h = 2 \, F$ gives the first claim.

Suppose now that $\frac{1}{2} < F(t) = \frac{1}{2} \, h(t)$ on $[t_0 , t_1]$ and $t_1 < t_0+\log \,\frac{2\,F(t_0)}{2\,F(t_0)-1}$.  It follows that (B) applies on this interval and  Lemma \ref{l:FTCsup} gives that
 \begin{align}
 	\log \frac{h(t_1)-1}{h(t_1)} \leq (t_1 - t_0) + \log \frac{h(t_0)-1}{h(t_0)}\, .
 \end{align}
 Exponentiating this and using that $h = 2 \, F$ gives the second claim (the assumption that $t_1 < t_0+\log \,\frac{2\,F(t_0)}{2\,F(t_0)-1}$ gives that 
a denominator is positive, which is used to preserve an inequality when   dividing).
\end{proof}

 \section{Eigenvalue evolution}

In this section, we will assume that $(M,g,f)$ satisfies \eqr{e:gf1} and \eqr{e:gf2}.
The $f$-divergence of a vector field $V$ is defined to be
\begin{align}	
	\dv_f \, V = \dv (V) - \langle V , \nabla f \rangle \, .
\end{align}
We will need the commutator of $\partial_t $ and $\cL$:

\begin{Lem}	\label{l:tcLu}
We have $\partial_t \, (\cL \, u ) = \cL \, u_t - 2\, \dv_f \, (\phi (\nabla u))$.
\end{Lem}

\begin{proof}
Since $g_{ij} g^{jk} = \delta_{ik}$, we have that $(g^{-1})' = - g^{-1} \, g_t \, g^{-1}$.  Therefore, 
since $\langle \nabla u , \nabla f \rangle = g^{ij} \, u_i \, f_j$, we have
\begin{align}	\label{e:tL1}
	\partial_t \, \langle \nabla u , \nabla f \rangle = - g_t (\nabla u , \nabla f)  + \langle \nabla u_t , \nabla f \rangle + \langle \nabla u , \nabla f_t \rangle \, .
\end{align}
Next, recall that $\Delta \, u = (\det \, g)^{- \frac{1}{2}} \, \partial_i \, \left( (\det \, g)^{\frac{1}{2}} g^{ij} \, u_j \right)$.   
To shorten notation below, set $\chi = \sqrt{ \det g}$.
Using that
\begin{align}
	 (\det \, g)' = (\det \, g) \, g^{ij} \, (g_t)_{ij} = 2\, (\det \, g) \, f_t \, ,
\end{align}
we see that $\chi' =  \chi \, f_t$ and $\left[ \chi^{-1} \right]' = - \chi^{-1} \, f_t$.
Using this gives
\begin{align}
	\partial_t \, \Delta \, u - \Delta \, u_t &=
	-  \chi^{-1}  \, f_t \, \partial_i \, \left( \chi \, g^{ij} \, u_j \right)
	+ \chi^{-1}  \, \partial_i \, \left( \chi \, f_t\, g^{ij} \, u_j \right) 
	- \chi^{-1} \, \partial_i \, \left( \chi \, g^{ip}\, (g_t)_{pq} \, g^{qj} \, u_j \right)  \notag  \\
	&=  \partial_i (f_t) \, g^{ij} \, u_j -  \chi^{-1} \, \partial_i \, \left( \chi \, g^{ip}\, (g_t)_{pq} \, g^{qj} \, u_j \right)  \\
	&= \langle \nabla f_t , \nabla u \rangle -  \dv \, ( g_t (\nabla u)) 	\notag \, ,
\end{align}
where $g_t (\nabla u)$ denotes the vector field dual to the one form $g_t (\nabla u, \cdot)$.  Combining this with \eqr{e:tL1} gives
\begin{align}
	\partial_t \, (\cL \, u)  &= \partial_t \, (\Delta \, u) - \partial_t \, \left( \langle \nabla u , \nabla f \rangle \right) \notag \\
	&= \Delta \, u_t 
	+ \langle \nabla f_t , \nabla u \rangle -  \dv \, ( g_t (\nabla u)) 
	  + g_t (\nabla u , \nabla f)  - \langle \nabla u_t , \nabla f \rangle - \langle \nabla u , \nabla f_t \rangle \\
	&= \cL \, u_t - \dv_f \, (g_t (\nabla u))
	\notag \, .
\end{align}
\end{proof}

In the remainder of this section, we will assume that $M$ is compact.  The arguments generalize to the non-compact case under suitable hypotheses to guarantee 
that various integrals converge and boundary terms vanish asymptotically.

\begin{Lem}	\label{l:dB}
If $u \in W^{1,2}$ and $\phi = \frac{1}{2} \, g - \Hess_f - \Ric$, then
\begin{align}
	\int \phi (\nabla u , \nabla u) \, \e^{-f} &= \int \left( |\Hess_u|^2 + \frac{1}{2} \, |\nabla u|^2 - (\cL \, u)^2 \right) \, \e^{-f} 		\, .
\end{align}
\end{Lem}

\begin{proof}
The drift Bochner formula gives
\begin{align}
	\frac{1}{2} \, \cL \, |\nabla u|^2  & =  |\Hess_u|^2 + \langle   \nabla \, \cL \, u , \nabla u \rangle + \Ric (\nabla u , \nabla u) + \Hess_f (\nabla u , \nabla u) \notag \\
	&=  |\Hess_u|^2 + \langle   \nabla \, \cL \, u , \nabla u \rangle - \phi (\nabla u , \nabla u) + \frac{1}{2} \, |\nabla u|^2 \, .
\end{align}
Using that $\cL \, |\nabla u|^2$ integrates to zero against $\e^{-f}$ and then integrating by parts gives the first equality in the lemma.
\end{proof}

\begin{Lem}	\label{l:evolves}
If   $u_t = \cL \, u + \frac{1}{2}  \, u$ and $v_t = \cL \, v +  \frac{1}{2} \, v$, then
\begin{align}
	\partial_t \, \int u\, v \, \e^{-f} &=    \int \left(  u \, v - 2 \, \langle \nabla u , \nabla v \rangle \right)  \, \e^{-f} \, , \\
	\partial_t \, \int |\nabla u|^2 \, \e^{-f} &= - 2\, \int     |\Hess_u|^2   \, \e^{-f} \, .
\end{align}
In particular, if $\int u \, \e^{-f} = 0$ at some time, then it vanishes at all future times.
\end{Lem}

\begin{proof}
The evolution preserves the volume element $\e^{-f} \, dv$, so we get for any function $w$ that
\begin{align}	\label{e:fvolp}
	\partial_t \, \int w \, \e^{-f} = \int  w_t \, \e^{-f} \, .
\end{align}
 Applying this with $w=u\, v$ where
\begin{align}
	(uv)_t = u\, v_t + v \, u_t = v \, ( \cL \, u +  \frac{1}{2} \, u) + u \, ( \cL \, v +  \frac{1}{2} \, v) = u \, v + \cL \, (u\, v) - 2 \, \langle \nabla u , \nabla v \rangle \, . \notag
\end{align}
Integrating this gives the first claim.
Using that 
\begin{align}
	|\nabla u|^2 = g(\nabla u , \nabla u) = g^{-1} (du , du) \, , 
\end{align}
we get
\begin{align}	\label{e:e216}
	\partial_t \, \int |\nabla u|^2 \, \e^{-f} &= \int \left( 2 \, \langle \nabla u , \nabla u_t \rangle - g_t (\nabla u , \nabla u) \right) \, \e^{-f} \notag \\
	&= - 2\, \int \left(   u_t \, \cL \, u + \phi  (\nabla u , \nabla u) \right) \, \e^{-f} \, ,
\end{align}
where the last equality used integration by parts and 
  $g_t = 2\, \phi$.  Applying Lemma \ref{l:dB}    gives
\begin{align}
	\partial_t \, \int |\nabla u|^2 \, \e^{-f} &= - 2\, \int \left( u_t \, \cL \, u + |\Hess_u|^2 + \frac{1}{2} \, |\nabla u|^2 - (\cL \, u)^2 \right) \, \e^{-f} \, .
\end{align}
The second claim follows from this and  $u_t = \cL \, u +  \frac{1}{2} \, u$.  Finally, applying \eqr{e:fvolp} with $w=u$ gives
\begin{align}
	\partial_t \, \int u \, \e^{-f} = \int  u_t \, \e^{-f} = \int (\cL \, u + \frac{1}{2} \, u ) \, \e^{-f} = \frac{1}{2} \, \int u \, \e^{-f} \, .
\end{align}
Integrating this, we see that if $\int u \, \e^{-f}$ vanishes at some time, then it also vanishes at all later times.
\end{proof}

 \subsection{Evolution of energy and inner products}

Given $u, v$ with $u_t = \cL \, u + \frac{1}{2}  \, u$ and $v_t = \cL \, v +  \frac{1}{2} \, v$, we will use the following quantities (motivated  by section $3$ in \cite{CM3}) repeatedly:
\begin{align}
	J_{uv} (t) &= \int u \, v \, \e^{-f} \, , \label{e:hereJ}  \\
	D_{uv} (t) &= \int \langle \nabla u , \nabla v \rangle \, \e^{-f} \, . \label{e:hereD}
\end{align}
We also define $I_u (t) = J_{uu} (t)$,  $E_{u} (t) = D_{uu} (t)$, and $F_u (t) = \frac{E_u (t)}{I_u(t)}$.   Using this notation, Lemma \ref{l:evolves} gives 
\begin{align}
	J_{uv}'(t) &=  J_{uv} (t) - 2 \, D_{uv}(t) \, ,  \label{e:IJ1} \\
	I_{u}'(t) &=  I_u(t) - 2\, E_u (t) \, , \label{e:IJ2} \\
	(\log I_u)'(t)&=1-2\,F_u(t)\,  , \label{e:log}\\
	E_u'(t) &= -2 \, \int  \left| \Hess_{u} \right|^2 \, \e^{-f} \leq 0 \, , \label{e:IJ3} \\
	F_u'(t) &= - 2\, \frac{ \int  \left| \Hess_{u} \right|^2 \, \e^{-f}}{I_u(t)} + F_{u} (t) \, \left( 2 \, F_{u}(t) - 1 \right) 
	\, . \label{e:IJ4}
\end{align}

\begin{proof}[Proof of Theorem \ref{l:compk}]
Translate in time so that $t_0 = 0$ and let $\bar{u}_1 , \dots , \bar{u}_k$ be $L^2$ orthonormal with at time $0$ with $\cL \, \bar{u}_i = - \lambda_i (0) \, \bar{u}_i$.  Integrating
 \begin{align}
 	\cL \, (\bar{u}_i \, \bar{u}_j) = (\lambda_i (0) + \lambda_j (0)) \, \bar{u}_i \, \bar{u}_j +
2 \, \langle \nabla \bar{u}_i , \nabla \bar{u}_j \rangle \, , 
\end{align}
 we see that 
\begin{align}	\label{e:baruizero}
	2\, \int   \langle \nabla \bar{u}_i , \nabla \bar{u}_j \rangle\, \e^{-f} &= - (\lambda_i (0) + \lambda_j (0)) \, \int \bar{u}_i \, \bar{u}_j  \, \e^{-f} = 0 \, .
\end{align}
Let $u_i (x,t)$ be the solutions of
$\partial_t \, u_i = \cL \, u_i + \frac{1}{2} \, u_i$ with $u_i (x,0) = \bar{u}_i (x)$.  Since $\int \bar{u}_i \, \e^{-f} = 0$,  the first claim in Lemma \ref{l:evolves} 
(with $u=u_i$ and $v=1$) gives that
\begin{align}
	\int u_i \, \e^{-f} = 0 {\text{ for all }} t \geq 0 \, .
\end{align}
Define 
functions $I_i(t), J_{ij} (t), E_i(t), D_{ij} (t)$ as in \eqr{e:hereJ}, \eqr{e:hereD}.  We have at $t=0$ that
\begin{align}
	I_i (0) = 1 , \, J_{ij} (0) = \delta_{ij} , \, E_i (0) =  \lambda_i (0)  {\text{ and }}  D_{ij} (0) = 
	\lambda_i (0) \, \delta_{ij} \, .
\end{align}
  Lemma \ref{l:evolves}  (see \eqr{e:IJ2}, \eqr{e:IJ2} and \eqr{e:IJ3}) gives
\begin{align} \label{e:Jij0} 
	I_i'(0)   = 1 - 2 \, \lambda_i \, ,  
	J_{ij}'(0)   = (1 - 2 \, \lambda_i ) \, \delta_{ij}  {\text{ and }}
	E_i'(0)  \leq 0 \, .
\end{align}
The $u_i$'s are orthonormal at $t=0$ and vary continuously, so the Gram-Schmidt process constructs a matrix $a_{ij} (t)$ so that
\begin{enumerate}
\item For each $t$, the functions $v_i (x,t) = \sum_{j} a_{ij}(t) \, u_i (x,t)$ are orthonormal.
\item $a_{ij} (0) = \delta_{ij}$.
\item  If $j > i$, then $a_{ij} (t) = 0$.
\end{enumerate}
Property ($1$) gives for all $t$ that 
\begin{align}
	\sum_{k,m} \, a_{ik}(t) \, J_{km}(t) \, a_{jm} (t)= \delta_{ij} \, .
\end{align}
  Differentiating this in $t$, then using \eqr{e:Jij0},  $a_{ij} (0) = \delta_{ij}$ by ($2$), and $J_{ij} (0) = \delta_{ij}$ gives 
\begin{align}
	0 &= \sum_{p,m} \, \left( a_{ip}'(0) \, \delta_{pm} \, \delta_{jm} + \delta_{ip} \, J_{pm}'(0) \, \delta_{jm} + \delta_{ip} \, \delta_{pm} \, a_{jm}'(0) \right) \notag 
	\\
	& = a_{ij}'(0) + (1 - 2 \, \lambda_i ) \, \delta_{ij} + a_{ji}'(0) \, .
\end{align}
It follows that $a_{ii}'(0) = \frac{1}{2} \,  ( 2 \, \lambda_i -1)$ and, using also ($3$) that $a_{ij}'(0)=0$ for $i \ne j$.  Observe that
\begin{align}
	 \int |\nabla v_i (\cdot , t)|^2 \, \e^{-f} = \sum_{j,k} a_{ij} (t) \, a_{ik}(t) \, D_{jk}(t) \, .
\end{align}
Differentiating this at $t=0$ and using that $a_{ij} (0) = \delta_{ij}$ and $D_{ij} (0) = \lambda_i (0) \, \delta_{ij}$ gives
\begin{align}
	\frac{d}{dt} |_{t=0} \,  \int |\nabla v_i (\cdot , t)|^2 \, \e^{-f} &= \sum_{j,k} 
	\left( a_{ij}' (0) \, \delta_{ik} \,  \lambda_j (0) \, \delta_{jk}(0) + \delta_{ij}  \, a_{ik}'(0) \, \lambda_j (0)  \delta_{jk} + \delta_{ij} \, \delta_{ik}\, D_{jk}'(0) \right) \notag \\
	&= 2\, a_{ii}'(0) \, \lambda_i (0) + F_{ii}'(0) = ( 2 \, \lambda_i (0) -1) \, \lambda_i (0) - 2 \, \int \left| \Hess_{\bar{u}_i} \right|^2 \, \e^{-f} \, .
\end{align}
Since the functions $v_i$ are orthonormal (and $\int v_i \, \e^{-f}=0$), so they give test functions for $\lambda_k$.  Therefore, in the sense of 
 \eqr{e:lsup}, we get that
\begin{align}
	\lambda_k' (0) &\leq \sup_{\{ i \leq k \, | \, \lambda_i (0) = \lambda_k(0) \} }\, \, \partial_t |_{t=0} \, \int |\nabla v_i|^2 \, \e^{-f} 
	\leq    ( 2 \, \lambda_k(0) -1) \, \lambda_k (0) \, .
\end{align}
We proved this at $0$, but the same arguments applies at each $t$ to give that
\begin{align}
	\lambda_k' (t) \leq  ( 2 \, \lambda_k(t) -1) \, \lambda_k (t)
\end{align}
in the sense of \eqr{e:lsup}.  Therefore, the first claim in Lemma \ref{l:diffineq} gives the first claim in the theorem.   The second claim in the theorem follows from Lemma \ref{l:ftc}.  

Finally, suppose that $\lambda_k (t_0) > \frac{1}{2}$ and  $t<t_0+\log \,\frac{2\,\lambda_k(t_0)}{2\,\lambda_k(t_0)-1}$.  We must show that
\begin{align}
	 \lambda_k (t) \leq \frac{\lambda_k(t_0)}{2\,\lambda_k(t_0)\,(1-\e^{t-t_0})+\e^{t-t_0}} \, .
\end{align}
Since the right-hand side is greater than $\frac{1}{2}$, this follows immediately from the second case if there is any $s \in [t_0 , t_1]$ with 
$\lambda_k (s) \leq \frac{1}{2}$.  Therefore, we can assume that $\lambda_k (s) > \frac{1}{2}$ for all $s \in [t_0 , t_1]$.  The last claim in the theorem now follows from the second claim in Lemma \ref{l:diffineq}.
\end{proof}

\begin{proof}[Proof of Theorem \ref{c:lambda1bd}]
Let $(M, \bar{g} , \bar{f})$ be the limiting gradient shrinking Ricci soliton and $\bar{\cL}$ its drift Laplacian.
By   \cite{CxZ} (cf. \cite{HN}), $\bar{\cL}$ has discrete spectrum   
\begin{align}
	\bar{\lambda}_0 = 0<   \frac{1}{2} < \bar{\lambda}_1  \leq \dots \to \infty \, , 
\end{align}
 the eigenfunctions   are in $W^{1,2}$, and (cf. \cite{CM2}) 
 \begin{align}
 	\lambda_1 (t_j) \to \bar{\lambda}_1 \, .
 \end{align}
 The theorem now follows immediately from the first claim in Theorem \ref{l:compk}.
\end{proof}

    \subsection{Sharpness}
The next example shows that the estimates in Theorem \ref{l:compk}
 are sharp.
Define the function $f$ and metric $g$ on $\RR^n$ by
\begin{align}	\label{e:fgont}
	f(x,t) = \frac{|x|^2}{4}- \frac{n}{2} \, \log u(t) {\text{ and }} g(t) = u(t) \, \delta_{ij} \, , 
\end{align}
where the function $u(t)$ is given by
\begin{align}	\label{e:trivsol}
	u(t) = 1 + \left(u(t_0) - 1 \right) \, \e^{t-t_0} \, .
\end{align}
It is easy to see that  $f$ and $g$ satisfy the modified Ricci flow equations \eqr{e:gf1} and \eqr{e:gf2}.  
Notice that if $u(t_0) > 1$, then $u$ is growing and is defined for all $t$ (i.e., it is eternal).  On the other hand, 
when $u(t_0) < 1$, then $u$ shrinks to zero in finite time - but the solution is ancient.

\begin{Lem}	\label{l:lam1}
If $g = a \, \delta_{ij}$ and $f = \frac{|x|^2}{4}$ on $\RR^n$, then $\lambda_1 =\frac{a^{-1}}{2}$.
\end{Lem}

\begin{proof}
Given a function $v$, we have $(\nabla v)^i = g^{ij} v_j$ and, thus, 
\begin{align}
	\langle \nabla v , \nabla f \rangle = g_{ik} (g^{ij} v_j)(g^{km} f_m) = g^{ij} v_i f_j = a^{-1} \, v_i \frac{x_i}{2} \, .
\end{align}
From this and the fact that $\partial_j$ commutes with $\Delta$ (since the metric is constant in space), we see that
\begin{align}
	\partial_j \, (\cL \, v) = \cL \, v_j - \frac{a^{-1}}{2} \, v_j \, .
\end{align}
It follows from the standard argument that the $L^2$ eigenvalues of $\cL$ occur at multiples of $\frac{a^{-1}}{2}$ and are given by polynomials.
Since $\Delta v = a^{-1} \, v_{ii}$, we see that $\Delta \, x_i =0$ and, thus, 
\begin{align}
	\cL \, x_i = - \langle \nabla x_i , \nabla f \rangle = - a^{-1} \, \delta_{ij} \, f_j = - \frac{a^{-1}}{2} \, x_i \, .
\end{align}
The $x_i$'s will be the lowest eigenfunctions (after the constant), giving the lemma.
\end{proof}

\begin{Cor}	\label{c:sharp}
The first eigenvalue $\lambda_1 (t)$ of the solution \eqr{e:fgont}, \eqr{e:trivsol} satisfies
\begin{align}	\label{e:sharp}
	\lambda_1 (t) = \frac{\lambda_1 (t_0)}{2\, \lambda_1 (t_0)\left( 1 -  \e^{t-t_0}\right)   +  \e^{t-t_0}} \, .
\end{align}
If $u(t_0) > 1$, then we get an eternal modified Ricci flow where $\lambda_1 (t) < \frac{1}{2}$ for all $t$.
\end{Cor}

\begin{proof}
It follows from  Lemma \ref{l:lam1} and \eqr{e:fgont} that 
\begin{align}
	\lambda_1 (t) = \frac{1}{2\, u(t)} \, . 
\end{align}
  In particular, $u(t_0) =  \frac{1}{2\, \lambda_1(t_0)}$ and, thus,  \eqr{e:trivsol} gives that
\begin{align}
	\lambda_1 (t) = \frac{1}{2\, u(t)} =  \frac{1}{2\,  \left\{ 1 + \left(u(t_0) - 1 \right) \, \e^{t-t_0} \right\}}
	=    \frac{1}{2\,  \left\{ 1 + \left( \frac{1}{2\, \lambda_1(t_0)} - 1 \right) \, \e^{t-t_0} \right\}} \, .
\end{align}
Simplifying this gives \eqr{e:sharp}.
\end{proof}

\section{Splitting theorem}

We will prove the splitting theorem, Theorem \ref{P:highsplit}, under the assumption that the   integrations by parts are justified and the   integrals converge.  
 For instance, when $M$ has finite weighted volume and a log Sobolev inequality, then Proposition $1$ in \cite{CxZ} (cf. \cite{HN}) guarantees that $\cL$ has discrete eigenvalues going to infinity with finite multiplicity and with eigenfunctions in $W^{1,2}$.

\vskip1mm
We start with  a more precise statement of the splitting:
 
\begin{Thm}	\label{P:highsplitA}
If   $\lambda_k (t_0) = \frac{1}{2}$ for some $k \geq 1$ and $\lambda_1 (t_1) \geq \frac{1}{2}$ with $t_0 < t_1$, then
 $M = N^{n-k} \times \RR^k$ and there are metrics $g_N(t)$ on $N$ and functions $f_N: N \times \RR \to \RR$  so that:
\begin{align}
 g (t) &= g_N (t) + \sum_{i=1}^k dx_i^2 \, , \\
 f &= f_N + \frac{1}{4} \, \sum_{i=1}^k x_i^2 \, , \\
	\partial_t \, g  & = \partial_t \, g_N = g_N - 2 \, \Hess_{\bar{f}} - 2 \, \Ric_N \, , \label{e:check1} \\
	f_t &= \bar{f}_t = \frac{n-k}{2} - S_N - \Delta_N \, \bar{f} \, .  \label{e:check2}
\end{align}
\end{Thm}

\begin{proof} 
Theorem \ref{l:compk} gives that $\lambda_k(t) \leq \frac{1}{2}$ for every $t \geq t_0$; since $\lambda_1 (t_1) = \frac{1}{2}$, we see that
\begin{align}	\label{e:lamit}
	\lambda_i (t) = \frac{1}{2} {\text{ for every }} t \in [t_0 , t_1] {\text{ and }} 1 \leq i \leq k \, .
\end{align}
 For $i= 1 , \dots , k$, let $u_i$ satisfy 
 \begin{align}
 	\partial_t u_i = \cL \, u_i + \frac{1}{2} \, u_i {\text{  for }} t \geq t_0
\end{align}
 with $u_i (\cdot , t_0)=\bar{u}_i (\cdot)$ equal to the $i$-th  eigenfunction   at time $t_0$ normalized so that $\int \bar{u}_i \, \bar{u}_j^2 \, \e^{-f} = \delta_{ij}$.  Note that $\int \bar{u}_i \, \e^{-f} = 0$ and   $\int \langle \nabla \bar{u}_i , \nabla \bar{u}_j \rangle \, \e^{-f} = \frac{1}{2} \, \delta_{ij}$.
Let $I_i, J_{ij} , E_i , D_{ij}, F_i$ be as in \eqr{e:hereJ}, \eqr{e:hereD}.  
  Lemma \ref{l:evolves} gives that $\int u_i \, \e^{-f} = 0$, $E_i'(t)  = -2 \, \int  |\Hess_{u_i}|^2 \, \e^{-f}$, 
    \begin{align}
  	J_{ij}'(t) &=   J_{ij} (t)  - 2\, D_{ij}(t)  \, ,  \label{e:Jija2}  \\
	F_i'(t) & = \frac{ -2 \, \int  |\Hess_{u_i}|^2 \, \e^{-f} }{I(t)} + F_i'(t) \, (2\, F_i (t) - 1) \leq F_i'(t) \, (2\, F_i (t) - 1)  \, . \label{e:numba2}
  \end{align}
  Since $F_i (t_0) = \frac{1}{2}$ and $F_i$ satisfies \eqr{e:numba2},
   the second claim in Lemma \ref{l:ftc} gives that $F_i (t) = \frac{1}{2}$ for every $t \geq t_0$.
   On the other hand,   $\int u_i \, \e^{-f} = 0$ for each $t$, so $u_i$ is a valid test function for $\lambda_1 (t) = \frac{1}{2}$ and, thus, we have that $F_i (t) \geq \frac{1}{2}$.  It follows that $F_i (t) \equiv \frac{1}{2}$ for every $t \geq t_0$ and, thus, each $u_i$ is an eigenfunction with eigenvalue $\frac{1}{2}$
   \begin{align}
   	\cL \, u_i = - \frac{1}{2} \, u_i {\text{ and }} \partial_t \, u_i = 0 {\text{ for every }} t \geq t_0 \, .
   \end{align}
   Using this and integrating by parts gives that 
   \begin{align}
   	J_{ij} (t) = -2 \, \int (\cL \, u_i) \, u_j \, \e^{-f} = 2 \, D_{ij}(t) \, .
\end{align}
  Using this in \eqr{e:Jija2}, it follows that 
   $J_{ij}' (t_0) = 0$.  Since $J_{ij} (t_0) = 0$ by construction,  it follows 
   \begin{align}	\label{e:DJij}
   	0 = J_{ij} (t) = 2\, D_{ij} (t) = 2 \, \int \langle \nabla u_i , \nabla u_j \rangle \, \e^{-f}  {\text{ for every }} t \geq t_0 \, .
\end{align}
Using again that $F_i (t) \equiv \frac{1}{2}$ in \eqr{e:numba2}, we also see that $\Hess_{u_i} \equiv 0$ and, thus, the vector fields $\nabla u_i$ are parallel for each $i$.  Since they are parallel, the functions
$\langle \nabla u_i , \nabla u_j \rangle$ are constant for each $t$; by \eqr{e:DJij}, we see that $\langle \nabla u_i , \nabla u_j \rangle \equiv 0$.   Since $I_i (t) = 2 \, E_i (t)$ is constant in time, 
we can arrange that 
$|\nabla u_i| \equiv 1$ after multiplying the $u_i$'s by a constant.

These $k$ parallel and orthonormal vector fields give the desired splitting
\begin{align}
	M = N \times \RR^k \, , 
\end{align}
where $N = \{ u_1 = \dots = u_k = 0 \}$ is the zero set of the $u_i$'s.  Since the $\nabla u_i$'s are parallel, they are also Killing fields and translation in $u_i$ is an isometry.  If follows that 
there is a family of metrics $g_N (t)$ on $N$ so that
\begin{align}
	g (t) = g_N (t) + \sum_i du_i^2 \, .
\end{align}
 Since $\Hess_{u_i} = 0$, so does $\Delta \, u_i$ and, thus, $\cL \, u_i = - \frac{1}{2} \, u_i$ implies that $\langle \nabla u_i , \nabla f \rangle  = \frac{1}{2} \, u_i$.  
 From this, the orthogonality of the $\nabla u_i$'s, and $|\nabla u_i| = 1$,  we see that for each $i$
 \begin{align}
 	\langle \nabla u_i  , \nabla \, \left( f - \frac{1}{4} \sum_j u_j^2 \right) \rangle = \langle \nabla u_i  , \nabla   f - \frac{u_i}{2}  \nabla  u_i  \rangle = 0 \, .
 \end{align}
 It follows that $\bar{f} = f - \frac{1}{4} \sum_j u_j^2$ does not depend on $\RR^k$ and, thus, depends only on $N$.    

It remains to verify \eqr{e:check1} and \eqr{e:check2}.  Using that $\partial_t u_i = 0$, $g_t = g - 2\, \Ric - 2 \, \Hess_f$ and   $\Hess_{\bar{f}} = \Hess_f - \frac{1}{2} \sum du_i^2$ gives \eqr{e:check1}.  Similarly, since $\partial_t u_i = 0$, 
$S = S_N$, and 
\begin{align}
	\Delta f = \Delta_N \, \bar{f} + \frac{1}{4} \, \sum_i \Delta \, u_i^2 =   \Delta_N \, \bar{f} + \frac{1}{2} \, \sum_i  |\nabla u_i|^2  =  \Delta_N \, \bar{f} + \frac{k}{2} \, ,
\end{align}
we have
\begin{align}
	\bar{f}_t &= f_t = \frac{n}{2} - S - \Delta \, f = \frac{n}{2} - S_N - \left( \Delta_N \, \bar{f} + \frac{k}{2} \right) \notag \\
	&= \frac{n-k}{2} - S_N - \Delta_N \, \bar{f} \, .
\end{align}
\end{proof}

\end{document}